\documentclass[12pt]{article}
\usepackage{amsmath}

\usepackage[utf8]{inputenc}
\usepackage[english]{babel}
\usepackage{amsmath, amsfonts, amsthm, amssymb, setspace, enumitem, tikz-cd, bbm, mathtools, graphicx,scalerel, mathrsfs, titling}

\DeclareMathOperator{\Spec}{Spec}

\DeclareMathOperator{\id}{id}
\DeclareMathOperator{\Gal}{Gal}

\DeclareMathOperator{\Hilb}{Hilb}
\DeclareMathOperator{\Hilbit}{\mathit{Hilb}}

\newcommand{\inv}{^{-1}}

\newcommand{\bN}{\mathbb{N}}
\newcommand{\bZ}{\mathbb{Z}}
\newcommand{\bQ}{\mathbb{Q}}
\newcommand{\bR}{\mathbb{R}}
\newcommand{\bC}{\mathbb{C}}
\newcommand{\bP}{\mathbb{P}}

\newcommand{\ra}{\rightarrow}

\newcommand{\catname}[1]{{\normalfont\textbf{#1}}}
\newcommand{\Set}{\catname{Set}}
\newcommand{\Sch}{\catname{Sch}}

\newtheoremstyle{break}
  {\topsep}{\topsep}%
  {\bfseries}{}%
  {\newline}{}%
\theoremstyle{break}
\newtheorem{theorem-break}{Theorem}
\newtheorem{lemma}{Lemma}
\newtheorem{theorem}{Theorem}
\newtheorem{corollary}{Corollary}
\newtheorem{proposition}{Proposition}

\theoremstyle{definition}

\newtheorem{example}{Example}
\newenvironment{pflabel}[1]{
	\begin{proof}[#1]
	}{
	\end{proof}
	}

\title{Abelian varieties over real closed fields}
\author{Nathanial Lowry}
\begin{document}

\begin{titlingpage}
    \maketitle
    \begin{abstract}
    In this paper, we classify the possible group
    structures on the set of $R$-valued points of an
    abelian variety, where $R$ is any real closed field.
    We make use of a family of abelian 
    varieties that, in effect, allows one to quantify  		
    over all abelian varieties of a fixed dimension and 
    degree of polarization in a first-order fashion.
    \end{abstract}
\end{titlingpage}

Frey and Jarden show in \cite{frey-jarden} that for every abelian variety of dimension $g>0$ defined over an algebraically closed field $K$, there is an isomorphism of groups $$A(K)\cong V\times(\bQ/\bZ)^{2g},$$ where $V$ is a $\bQ$-vector space of dimension equal to $|K|$, the cardinality of $K$. We prove an analogous result for real closed fields:

\begin{theorem}
\label{theorem-structure}
Let $R$ be a real closed field and $A/R$ an abelian variety of dimension $g>0$. Then there is an isomorphism of groups $$A(R)\cong V \times (\bQ/\bZ)^g\times (\bZ/2\bZ)^d$$ for some $0\leq g \leq d$, where $V$ is a $\bQ$-vector space of dimension $|R|$, the cardinality of $R$.
\end{theorem}

We first prove that $A(R)$ has the desired rank $\dim_{\bQ}A\otimes_{\bZ} \bQ$.

\begin{proposition}
\label{prop-freepart}
Let $A/R$ be an abelian variety of dimension $g>0$ over a real closed field. Then the rank of $A(R)$ is equal to the cardinality of $R$.
\end{proposition}
\begin{proof}
It is known that real closed fields are ``ample" (or ``large" as in Pop \cite{pop}) in the sense that for every smooth curve $C/R$, either $C(R)=\emptyset$ or $C(R)$ is infinite. Now \cite[Theorem 1.1]{ample} demonstrates that $A(R)$ has infinite rank.If $R$ countable, then so too is $A(R)$ and the result follows.\newline \indent If $R$ is uncountable, then one may employ the implicit function theorem for real closed fields (as in \cite{real}, Corollary 2.9.8) to find at least $|R|$ points in $A(R)$. Indeed, consider an affine open $U=\Spec(R[x_1,\dots, x_n]/(f_1,\dots, f_m))$ around the identity $e\in A(R)$. As $A$ is smooth, it is a local complete intersection, so we may assume that $n-m=g$. Consider the map $R^n\ra R^m$ given by the $f_i$. Again since $A$ is smooth, the Jacobian matrix $(\partial f_i / \partial x_j)$ has maximal rank $m$ at $e$ and thus there is some $m\times m$ minor of this matrix which is non-vanishing. After relabeling coordinates we may assume that $e$ belongs to the open set where $\det( (\partial f_i / \partial x_j)_{g+1\leq j\leq n})\neq 0$. Now the implicit function theorem guarantees the existence of a non-empty open subset $C\subset R^g$ and a definable function $h:C\ra R^m$ so that $(a_1,\dots, a_g, h(a_1,\dots, a_g))\in A(R)$ for every $(a_1,\dots, a_g)\in C$. As non-empty open subsets in $R^g$ have cardinality $|R|$, this demonstrates at least $|R|$ points of $A(R)$. As the cardinality of $A(R)$ is certainly at most that of $R$, we have $|A(R)|=|R|$. \newline \indent The torsion subgroup $T\subset A(R)$ is at most countable, so $|A(R)/T|=|R|$ since we assume $R$ to be uncountable. Finally, $A(R)/T$ is isomorphic to a subgroup of $A(R)\otimes \bQ$, so $|A(R)\otimes \bQ|=|R|$ as well.
One easily shows that for any infinite cardinal $\kappa$, the $\bQ$-vector space $\oplus_{i\in\kappa} \bQ$ has cardinality $\kappa$, so the result follows.
\end{proof}

\begin{proposition}
\label{torsion-prop}
Let $R$ be a real closed field, $A/R$ an abelian variety of dimension $g$, $n$ a positive integer, and $p$ a prime. For $p\neq 2$, there are exactly $p^{ng}$ points of order dividing $p^n$ on $A$ which are defined over $R$, i.e. $|A[p^n](R)|=p^{ng}$. For $p=2$, $|A[2^n](R)|=2^{ng+d}$ for some $0\leq d\leq g$ which is independent of $n$.
\end{proposition}

One might hope to employ the fact that real closed fields are elementarily equivalent to $\bR$ in the sense that the truth in $R$ of any first-order sentence in the language of fields is equivalent to its truth in $\bR$. In the concrete case of elliptic curves, there is a motivating example indicating that this strategy may be possible. 

\begin{example}
\label{example}
Every elliptic curve over a field $K$ of characteristic not equal to 2 or 3 can be realized as the locus in $\bP^2$ of a polynomial $Y^2Z=X^3+AXZ^2+BZ^3$ with $4A^3+27B^2\neq 0$ and with $[0:1:0]$ serving as the identity. These equations describe a family of elliptic curves (i.e. an abelian scheme) $\mathcal{E}\ra\Spec(\bZ[x,y,1/(4x^3+27y^2)])$ that, in effect, allows us to quantify over all elliptic curves in a first-order fashion. Moreover, the condition that a point have order dividing $p^n$ can be expressed by certain polynomial identities with parameters depending uniformly on $A$ and $B$ so that the statement ``every elliptic curve has exactly $p^n$ points of order dividing $p^n$" really is first-order statement in the language of fields. For example, the following first-order statement is equivalent to the statement that every elliptic curve over $K$ has exactly 2 points of order 2: $$\forall a,b (4a^3+27b^2\neq 0 \ra \exists x (x^3+ax+b=0,  \forall z (z^3+az+b=0\rightarrow x=z)))$$
\noindent In general one has to work with projective coordinates, but this is a quotient of $K^n$ by a definable equivalence relation so it is a non-issue.
\end{example}

\noindent This strategy seems more viable once we establish the truth of Proposition \ref{torsion-prop} over $\bR$.

\begin{lemma}
\label{reals}
Proposition \ref{torsion-prop} holds for $R=\mathbb{R}$.
\end{lemma}
\begin{proof}
In fact, much more is known in this case. It is shown in \cite{gross-harris} that there is an isomorphism of real Lie groups $A(\bR)\cong (\bR/\bZ)^g\times (\bZ/2\bZ)^d$ for some $0\leq d\leq g$. The torsion subgroup $(\bQ/\bZ)^g\subset (\bR/\bZ)^g$ is divisible (or equivalently,  injective), so we get an isomorphism of groups $A(\bR)\cong V\times (\bQ/\bZ)^g\times (\bZ/2\bZ)^d$ where $V$ is a uniquely divisible group (i.e. a $\bQ$-vector space) of rank equal to the cardinality of $\bR$.

\end{proof}

We now seek a family of abelian varieties akin to that of Example \ref{example} which will allow for a first-order universal quantification over all abelian varieties of a fixed dimension and degree of polarization.


\begin{lemma}
\label{family}
For every pair of positive integers $g$ and $d$, there is a finite collection of abelian schemes $A_i\rightarrow S_i$ whose bases $S_i$ are quasi-projective (and therefore finite type) over $\Spec(\bZ)$ that satisfy the following condition: for every field $K$ and abelian variety $B/K$ of dimension $g$ possessing a polarization of degree $d^2$, there is a $K$-point $s:\Spec(K)\rightarrow S_i$ for some $i$ so that the fiber $(A_i)_s/K$ is isomorphic to $B$ as an abelian variety over $K$.
\end{lemma}
\begin{proof} For any fixed rational polynomial $\Phi$, recall that the functor $$\Hilbit(n,\Phi):\Sch^{op}\ra \Set$$ given by 
\begin{align*}
&\Hilbit(n,\Phi)(T)=\{Y\subset \mathbb{P}^n_T:& Y\subset \mathbb{P}^n_T\text { is closed with Hilbert polynomial } \Phi\\ &\text{ and }Y/T\text{ is flat}\}
\end{align*} is represented by a scheme $\Hilb(n,\Phi)$ which is projective over $\Spec(\bZ)$ (see \cite{FGAHilbert}, for instance). Let $Z/\Hilb(n,\Phi)$ be the universal object of this functor. Then $Z$ represents the functor $\Hilbit^*(n,\Phi):\Sch^{op}\rightarrow \Set$ which further specifies a section:
\begin{align*}
&\Hilbit^*(n,\Phi)(T)=\{(Y\subset \mathbb{P}^n_T, s: T\rightarrow Y) : Y\subset \mathbb{P}^n_T\text { is closed with Hilbert} \\&\text{ polynomial } \Phi,Y/T\text{ is flat, and } s\text{ is a section}\}
\end{align*}

\noindent Let $(Z'/Z, \varepsilon: Z\ra Z')$ be the universal object for this functor. Now, given an abelian variety $B/K$ of dimension $g$ possessing a polarization $\phi:B\ra B^{\vee}$ of degree $d^2$, the line bundle $(\id,\phi)^*(\mathcal{P})^3$ (where $\mathcal{P}$ is the Poincar\'e bundle on $B\times B^{\vee}$) induces an embedding $B\hookrightarrow \bP^n_K$ for which $n$ and the Hilbert polynomial $\Phi$ of $B$ are determined entirely by $g$ and $d$ (see \cite[Ch. 3, sec. 16]{mumfordAV}). We thus obtain a morphism $\Spec(K)\rightarrow Z$ (since abelian varieties come with sections). After restricting to the irreducible components of $Z$ and considering the smooth locus of $Z'\ra Z$, \cite[Theorem 6.14]{GIT} ensures that we have a finite collection of abelian schemes with the desired property.
\end{proof}

\begin{pflabel}{Proof of Proposition \ref{torsion-prop}}
It suffices to demonstrate this in the particular case that $A$ further possesses a polarization of degree $d^2$, as every abelian variety possesses such a polarization for some $d$. To verify that a point $P$ on $A$ has order dividing $p^n$ is equivalent to verifying certain polynomial identities, so that this statement being first order is evident from Lemma \ref{family}. Its truth, which can be verified over $\bR$, is evident from Lemma \ref{reals}.
\end{pflabel}

\begin{pflabel}{Proof of Theorem \ref{theorem-structure}}
It follows immediately from Proposition \ref{torsion-prop} and the fundamental theorem of finite abelian groups that $A[p^n](R)\cong (\bZ/p^n\bZ)^g$ and $A[2^n](R)\cong (\bZ/2^n\bZ)^g\times (\bZ/2\bZ)^d$. Since $\bQ/\bZ\cong \oplus_p \bQ_p/\bZ_p\cong \oplus_p \varinjlim \bZ/p^n\bZ$, the torsion subgroup $T$ is of the desired form. As this is a divisible group times a finite group, it follows from \cite[Theorem 8.1]{baer} that $A(R)$ is torsion-split, i.e. $A(R)\cong V\oplus T$, where $V$ is torsion free. We need only to verify that $V$ is uniquely divisible. First note that $2A(R)$ is uniquely divisible. Indeed, it is uniquely $n$-divisible for each $n$, as this is a first-order statement which can be verified over $\bR$. Now given any $v\in V$, $2v\in 2A(R)$ whence $2v/2n\in A(R)$, i.e. there is some $y\in A(R)$ so that $2v=2ny$. Writing $y=w+t$ for $w\in V$ and $t\in T$, we find that $2v=2nw$, and since $V$ is torsion-free we must have $v=nw$. That $V$ is \emph{uniquely} divisible follows since $2A(R)$ is uniquely divisible.
\end{pflabel}

\begin{corollary}
Let $R$ be a real closed field, $p$ an odd prime, and $A/R$ an abelian variety of dimension $g$. Then the Tate module $T_p(A)$ is a free $\bZ_p[G]$-module of rank $g$ where $G=\Gal(R(\sqrt{-1})/R)$.
\end{corollary}
\begin{proof}
This follows immediately from the description of $A[p^n](R)$ above.
\end{proof}

The following result is well-known, but we give an (ostensibly) new proof using the family of abelian varieties constructed in Lemma \ref{family}.
\begin{corollary}
Let $A/k$ be an abelian variety of dimension $g$ over an algebraically closed field $k$ of characteristic 0. For every positive integer $n$ we have $A[n]\cong (\bZ/n\bZ)^{2g}$.
\end{corollary}
\begin{proof}
This statement is first-order given Lemma \ref{family} and can be verified over $\bC$, where it is easily seen to be true.
\end{proof}

Returning briefly to Example \ref{example}, one may notice that the elliptic curves with \emph{connected} real locus occur over the fibers with $4A^3+27B^2>0$, and that those with \emph{disconnected} real locus occur over the fibers with $4A^3+27B^2<0$. We show that this is a more general phenomenon.
\begin{proposition}
\label{deformation}
Let $A\ra S$ be an abelian scheme of dimension $g$ with $S$ of finite type over a real closed field $R$. Then the function $S(R)\ra \bN$ sending $s\in S(R)$ to $|A_s[2](R)|$ is locally constant with respect to the order topology on $S(R)$ inherited from the ordering on $R$. In particular, the isomorphism class of the fiber (as an abelian group) is constant on the connected components of $S(R)$ (in the order topology).
\end{proposition}
\begin{proof}
We seek to apply a suitable version of the following elementary fact from the theory of real manifolds, whose proof (as in \cite{milnor}) is a simple application of the inverse function theorem: let $f:M\ra N$ be a smooth map between real manifolds of the same dimension with $M$ compact. If $U$ is the (open) subset of $N$ consisting of regular values, then the function $U\rightarrow \bN$ which sends $y\in U$ to $|f\inv(y)|$ is locally constant. 

Now, the 2-torsion subgroup scheme $A[2]$ is \'etale over $S$, as it is the pullback of the \'etale morphism $[2]:A\ra A$. For each $s\in S(R)$, the set $A_s[2](R)$ is finite and non-empty by Theorem \ref{theorem-structure}, so at the level of $R$-points, we have a surjective \'etale morphism $A[2](R)\ra S(R)$ (where an \'etale morphism of algebraic sets is a regular morphism which induces an isomorphism of tangent spaces). A definable version of the inverse function theorem is used in \cite[Proposition 8.1.2]{real} to conclude that $A[2](R)\ra S(R)$ is a local diffeomorphism. The first result now follows from the exact same argument as in the case of real manifolds. The isomorphism class of a real abelian variety (as an abelian group) is completely determined by its 2-torsion subgroup, so the second statement follows.

\end{proof}

\bibliography{thesis_results.bib}
\bibliographystyle{amsplain}

\end{document}